\documentclass[article]{amsart}
\usepackage{dsfont}
\usepackage{color}
\usepackage{cases}
\usepackage{amsmath, amsfonts, pifont, amssymb}
\usepackage{enumerate}
\usepackage{verbatim}
\usepackage[bookmarks=true]{hyperref}
\usepackage{geometry}
\newtheorem{theorem}{Theorem}[section]
\newtheorem{lemma}[theorem]{Lemma}
\newtheorem{proposition}[theorem]{Proposition}

\newcommand{\R}{\mathbb{R}}

\newcommand{\beq}{\begin{equation}}
\newcommand{\eeq}{\end{equation}}
\newcommand{\beqq}{\begin{equation*}}
\newcommand{\eeqq}{\end{equation*}}

\newcommand{\lf}{\left}
\newcommand{\ri}{\right}
\theoremstyle{definition}
\newtheorem{definition}[theorem]{Definition}

\theoremstyle{remark}
\newtheorem{remark}[theorem]{Remark}

\numberwithin{equation}{section}

\def\ep{\epsilon}

\usepackage{mathrsfs}
\newcommand{\T}{\mathbb{T}}
\newcommand{\Z}{\mathbb{Z}}
\newcommand{\N}{\mathbb{N}}
\newcommand{\C}{\mathbb{C}}

\newcommand{\ii }{{\rm i} }

\newcommand{\ps}[2]{\left( #1 ,#2 \right)}
\newcommand{\psc}[2]{\left\langle #1,#2 \right\rangle}
\newcommand{\p}{\partial}

\numberwithin{equation}{section}
\newcommand\bna{\begin{eqnarray*}}
\newcommand\ena{\end{eqnarray*}}
\newcommand\bnan{\begin{eqnarray}}%numérotée
\newcommand\enan{\end{eqnarray}}

\begin{document}

\address{Jingrui Niu
\newline \indent Sorbonne Université, CNRS, Université Paris Cité, Inria Team CAGE, Laboratoire Jacques-Louis Lions (LJLL), F-75005 Paris, France
}
\email{jingrui.niu@sorbonne-universite.fr}

\address{Zehua Zhao
\newline \indent Department of Mathematics and Statistics, Beijing Institute of Technology, Beijing, China.
\newline \indent Key Laboratory of Algebraic Lie Theory and Analysis of Ministry of Education, Beijing, China.}
\email{zzh@bit.edu.cn}

\title[Control for semi-periodic Schr\"odinger equations]{Observability and controllability for Schr\"odinger equations in the semi-periodic setting}
\author{Jingrui Niu and Zehua Zhao}

\begin{abstract}
Strichartz estimates, well-posedness theory and long time behavior for (nonlinear) Schr\"odinger equations on waveguide manifolds $\mathbb{R}^m \times \mathbb{T}^n$ are intensively studied in recent decades while the corresponding control theory and observability estimates remain incomplete. The purpose of this short paper is to investigate the observability and controllability for Schr\"odinger equations in the waveguide (semi-periodic) setting. 

Our main result establishes local exact controllability for the cubic nonlinear Schr\"odinger equations (NLS) on $\mathbb{R}^2 \times \mathbb{T}$, under certain geometric conditions on the control region. To address the nonlinear control problem, we begin by analyzing the observability properties of the linear Schr\"odinger operator on a general waveguide manifold $\mathbb{R}^m \times \mathbb{T}^n$. Utilizing $H^s$ estimates of the Hilbert Uniqueness Method (HUM) operator and Bourgain spaces, we then prove local exact controllability through a fixed-point method.
\end{abstract}


\maketitle

\setcounter{tocdepth}{1}
\tableofcontents

% \bigskip

% \noindent \textbf{Keywords}: nonlinear Schr\"odinger equations, waveguide manifolds, well-posedness, controllability, observability.
% \bigskip

% \noindent \textbf{Mathematics Subject Classification (2020)} Primary: 93B05; Secondary: 35R01, 93C20, 35Q55.
\section{Introduction}
In this paper, we study the internal exact controllability for the semi-periodic nonlinear Schr\"odinger equation on the waveguide manifold $\R^m\times\T^{n}$ (with dimension $d=m+n$): 
\begin{equation}\label{eq: NLS-intro}
\left\{
\begin{array}{ll}
\ii\p_t u+\Delta_{\R^m\times\T^n} u+\ep|u|^{p-1}u=f &\text{ on } [0,+\infty)\times\R^m\times\T^n  \\
 u|_{t=0}=u_0    & \text{ on }\R^m\times\T^n
\end{array}
\right.
\end{equation}
where $m$ and $n$ are two positive integers, and $\ep\in\R$. $\Delta_{\R^m\times\T^n}$ is the canonical Laplacian defined in $\R^m\times\T^n$, and sometimes we may write $\Delta_{x,y}$ for short. 
\subsection{Background and motivations}
In the literature, there has been extensive research on the controllability and observability of (nonlinear) Schr\"odinger equations on compact manifolds and bounded domains. We begin with the linear setting, particularly addressing observability problems. When the control region is an open set satisfying the geometric control condition (GCC), Lebeau \cite{Leb-Sch} proved that observability is true for an arbitrarily short time \(T>0\). However, GCC is not always a necessary condition. In various contexts, including the torus \cite{AM-JEMS,BBZ,BZ-MRL,Jaffard90}, compact hyperbolic surfaces \cite{Jin} (also see \cite{AR} for negatively curved manifolds), and the disk \cite{ALM}, observability has been established for any \(T>0\) and for any non-empty open control region, provided the control region encompasses a neighborhood of some portion of the boundary.  It is also worth mentioning that for certain subelliptic Schr\"odinger equations, the minimal observability time can be strictly positive (see \cite{Burq-Sun}).

Regarding the controllability of nonlinear Schr\"odinger equations (NLS), significant results have been obtained for different dimensions on compact manifolds or bounded domains (though the list here is not exhaustive). In one dimension, Laurent \cite{laurent2010global} proved global internal controllability in large time for NLS on a bounded interval, while local results were also achieved in \cite{RZ-1}. For two dimensions, Dehman, G\'erard, and Lebeau \cite{DGL} established exact controllability in $H^1$ for the defocusing NLS on compact surfaces. In the three-dimensional case, Laurent \cite{Laurent-dim3} demonstrated global internal controllability over large time intervals for NLS on certain compact manifolds. Concerning the boundary controls, we refer to \cite{Gagnon,RZ-09} and the references therein.  For a broader overview of these results, we refer to the surveys \cite{zuazua-survey,Laurent-survey}.

More recently, attention has shifted towards understanding observability for Schr\"odinger equations in non-compact settings \cite{HWW-CMP,WWZ-JEMS,prouff2023}. The specific case of periodic domains has been explored for the observability of the free Schr\"odinger equation \cite{taufer} and for settings with a periodic potential \cite{BM2023}. Notably, in \cite{BM2023}, the control region, although \(2\pi \mathbb{Z}^2\)-periodic, might not satisfy  GCC. Despite these advances, the study of controllability and observability in semi-periodic settings remains largely unexplored.
% (paragraphs for the observability and the control...)

We now provide a concise overview of our research focus: ``\textbf{(nonlinear) Schr\"odinger equations on waveguides}". This area has garnered significant attention over the past few decades, emerging as a prominent subject within nonlinear dispersive equations. Our approach blends traditional dispersive techniques with innovative analysis tools to delve into this intriguing domain.

%The waveguide manifolds $\R^m \times \T^n$ are a product of the Euclidean space with tori, and are of particular interest in nonlinear optics of telecommunications. Presently, within backbone networks, data signals predominantly travel via optical carriers through fibers, which serve as specialized waveguides. With applications such as the internet requiring expanded bandwidth and cost-effective data transmission, there's a growing emphasis on optimizing network capabilities. The nonlinear Schr\"odinger model holds significant relevance in understanding nonlinear phenomena within optical fibers, crucial for enhancing their performance and efficiency. Drawing from physics, an optical waveguide serves as a conduit that directs light waves along a predetermined trajectory. A fascinating aspect of examining solution behaviors on the waveguide manifold is its amalgamation of properties inherited from both classical Euclidean spaces and tori, providing a comprehensive understanding of its underlying physics. 

Waveguide manifolds, denoted as $\R^m \times \T^n$, represent the product of Euclidean space with tori and are of particular relevance in nonlinear optics, especially within telecommunications. Currently, data signals in backbone networks predominantly travel via optical carriers through fibers, which function as specialized waveguides. As applications like the internet demand greater bandwidth and cost-effective data transmission, there is increasing emphasis on optimizing these network infrastructures. The nonlinear Schr\"odinger equation plays a pivotal role in modeling nonlinear effects in optical fibers, essential for enhancing performance and efficiency. In physics, an optical waveguide directs light along a defined path, and the study of solutions on waveguide manifolds is particularly intriguing due to the unique combination of properties inherited from both Euclidean spaces and tori, offering deeper insights into the underlying physics.

Given the nature of these combined spaces, the nonlinear Schr\"odinger Equation (NLS) posed on the waveguide manifold inherits characteristics from both Euclidean spaces and tori. The Euclidean case is studied and the theory, at least in the defocusing setting, is well established. (See \cite{CW,I-team1,Dodson3,Taobook} and the references therein.) Moreover, we refer to \cite{HTT1, IPT3, KV16, Yue} for a few works on tori. Due to the nature of such product spaces, we see NLS posed on the waveguide manifold mixed inheriting properties from those on classical Euclidean spaces and tori. The techniques used in Euclidean and tori settings are frequently combined and applied to waveguide problems. We refer to \cite{R2T,CGZ,CZZ,HP,HTT1,TV1,TV2,HTT2,IPT3,IPRT3,yang2023scattering,Z1,Z2,ZZ}  for some NLS results in the waveguide setting. At last, we note that, though scattering behavior is not expected for the periodic case because of the lack of dispersive, for some specific models of waveguides, modified scattering, and scattering results can be obtained as in the Euclidean. (See \cite{R2T,HP,TV2,MR3406826} for example.) 

\subsection*{Geometric setting}
We work in a $d-$dimensional waveguide manifold as $\R^m\times\T^n$, where $m$ and $n$ are two positive integers with (whole dimension) $d=m+n$ ($m,n \geq 1$). We consider a control region $\Omega$ of the following type~:

\smallskip {\bf (G)} $\Omega=(\Omega_1,\Omega_2)\subset\R^m\times\T^n$\footnote{This means $\Omega=\{(x,y)| x\in \Omega_1, y \in \Omega_2\}$.}. Let $\Omega_1\subset\R^m$ be a nonempty, open, $2\pi \mathbb{Z}^m$-invariant set. Let $\Omega_2 \subset \mathbb{T}^{d-m}$ be open and nonempty.

This kind of region may not satisfy the Geometric Control Condition, which is raised in e.g. \cite{Leb-Sch}

\subsection{The statement of the main results}
In this subsection, we present the main results of this paper respectively. We note that all of the results can be generalized to all dimensions in a natural way. Since 3D cubic NLS model is typical and popular\footnote{We refer to Kenig-Merle \cite{kenig2010scattering} for the scattering result in the Euclidean setting. See also the references therein.}, in Theorem \ref{thm1} and Theorem \ref{thm4}, we adopt 3D cubic NLS model\footnote{We consider cubic NLS on $\R^2\times\T$ or $\R\times\T^2$. Since the proofs work in the same way, we only consider the $\R^2\times\T$ case.} to illustrate the nonlinear results. 
\begin{theorem}[Exact controllability]\label{thm1}
Let $T>0$, and $\ep=\pm1$. Let $\Omega=(\Omega_1,\Omega_2)$ satisfy the condition (\textbf{G}). For any $s\geq 1$, there exists $\delta>0$, such that for all $u_0,u_{f}\in H^s(\R^2\times\T)$ satisfying that $\|u_0\|_{H^s(\R^2\times\T)}+\|u_f\|_{H^s(\R^2\times\T)}<\delta$, there exists a control function $f\in C([0,T];H^s(\R^2\times\T))$ supported in $[0,T]\times \Omega$ such that the unique solution $u\in X^{s,b}_T$ to
\begin{equation}
\left\{
\begin{array}{ll}
\ii\p_t u+\Delta_{x,y} u+\ep|u|^{2}u=f &\text{ on }[0,T]\times\R^2\times\T,  \\
 u|_{t=0}=u_0    & \text{ on }\R^2\times\T,
\end{array}
\right.
\end{equation}
fulfils $u|_{t=T}=u_f$.
\end{theorem} 
This is a local exact controllability result for NLS in the semi-periodic setting. The general strategy to establish the exact controllability is to reduce it to the \textit{null controllability}, i.e. the exact controllability with null final data ($u_{f}=0$). To be more specific, we first prove the following result.
\begin{theorem}[Null controllability]\label{thm4}
Let $T>0$. Let $\ep$ and $\Omega$ be the same as in Theorem \ref{thm1}. For any $s\geq 1$, there exists $\delta>0$, such that for all $u_0\in H^s(\R^2\times\T)$ satisfying that $\|u_0\|_{H^s(\R^2\times\T)}<\delta$, there exists a control function $g$ such that the unique solution $u\in X^{s,b}_T$ to
\begin{equation}\label{eq: NLS-u-g}
\left\{
\begin{array}{ll}
\ii\p_t u+\Delta_{x,y} u+\ep|u|^{2}u=\varphi_T\chi_{\Omega}(1-\Delta_{x,y})^{-s}\varphi_T\chi_{\Omega}g&\text{ on }[0,T]\times\R^2\times\T,  \\
 u|_{t=0}=u_0    & \text{ on }\R^2\times\T,
\end{array}
\right.
\end{equation}
fulfils $u|_{t=T}=0$. Furthermore, the control function $g$ verifies the following conditions
\begin{enumerate}
    \item $g\in C([0,T];H^{-s}(\R^2\times\T))$
    \item $\varphi_T(\cdot)=\varphi_1(\cdot/T)\in C^{\infty}(\R)$ where $\varphi_1(t)=1$ for $t\leq\frac{1}{2}$ and $\varphi_1(t)=0$ for $t\geq\frac{3}{4}$;
    \item $0\leq \chi_{\Omega}\in C^{\infty}(\R^2\times\mathbb{T})$ satisfies $\mathbf{1}_{\Omega'\times\Omega_2}\leq\chi_{\Omega}\leq\mathbf{1}_{\Omega_1\times\Omega_2}$, where  $\overline{\Omega'}\subset\Omega_1$.
\end{enumerate}
\end{theorem} 
To investigate the nonlinear case, a classic approach is to first study the linearized Schr\"odinger equation. The following theorem concerns solutions of the stationary Schr\"odinger equation and is applicable to high-energy eigenfunctions.
\begin{theorem}[Stationary estimate]\label{thm2}
Let $\Omega=(\Omega_1,\Omega_2)$ satisfy the condition (\textbf{G}). For all $\lambda \in \mathbb{R}$ we have the following estimate: for 
\begin{equation}
(\Delta_{x,y}-\lambda)u=f,    
\end{equation}
we have,
\begin{equation}\label{eq: est-unique-continuation}
\|u(x,y)\|_{L^2(\mathbb{R}^n \times \mathbb{T}^m)} \leq C(\|f(x,y)\|_{L^2(\mathbb{R}^n \times \mathbb{T}^m)}+\|u(x,y)\|_{L^2(\Omega)}),
\end{equation}
with $C$ independent of $\lambda$.
\end{theorem}
\begin{remark}
In particular, if $f=0$, this estimate implies a unique continuation property for the eigenfunctions of $-\Delta_{x,y}$. Indeed, if $f=0$, the total $L^2-$norm of $u$ can be bounded by its local $L^2-$norm in $\Omega$. As a consequence, if an eigenfunction $u$ vanishes in a subdomain $\Omega$, which verifies the condition (\textbf{G}), then, it vanishes everywhere. 
\end{remark}
Related to this stationary estimate, we have the following \textit{observability estimate} for linear Schr\"odinger propagator. Once we have it, using \textit{Hilbert uniqueness method} (HUM for short, see \cite{HUM} and Section \ref{sec: HUM}), we obtain the controllability result automatically for the linearized equation. 
\begin{theorem}[Observability]\label{thm3}
Let $\Omega=(\Omega_1,\Omega_2)$ satisfy the condition (\textbf{G}). For every $T>0$, there exists a constant $C=C(T,\Omega)$ such that for $\forall u_0\in L^2(\mathbb{R}^n \times \mathbb{T}^m)$,
\begin{equation*}
\|u_0\|^2_{L^2(\mathbb{R}^n \times \mathbb{T}^m)}\leq C\int_0^T\|e^{\ii t\Delta_{x,y}}u_0\|^2_{L^2(\Omega)}dt.
\end{equation*}
\end{theorem}
\subsection{Strategy of the proofs}

Now we discuss the strategy of the proofs for the main theorems stated in the previous subsection. 

Our primary objective is to prove Theorem \ref{thm1}, i.e., the local exact controllability for nonlinear Schr\"odinger equations on $\R^2\times\T$. We first reduce it to the null controllability with certain control in a specific form (Theorem \ref{thm4}). Following the classic approach, we need to analyze the control problem for the corresponding linear equation:
\begin{equation*}
\left\{
\begin{array}{ll}
\ii\p_t u+\Delta_{x,y} u=f &\text{ on }[0,T]\times\R^2\times\T,  \\
 u|_{t=0}=u_0, u|_{t=T}=0.   & \text{ on }\R^2\times\T,
\end{array}
\right.
\end{equation*}
As a consequence of \textit{Hilbert Uniqueness Method} (see Section \ref{sec: HUM} for details), constructing the control operator $\mathcal{L}: u_0\mapsto f$ is equivalent to proving observability for any $s\in\R$,
\begin{equation}\label{eq: H^s-ob-intro}
\|u_0\|^2_{H^s(\mathbb{R}^2 \times \mathbb{T})}\leq C\int_0^T\|\chi_{\Omega}e^{\ii t\Delta_{x,y}}u_0\|^2_{H^s(\mathbb{R}^2 \times \mathbb{T})}dt.
\end{equation}
In order to prove \eqref{eq: H^s-ob-intro} for any $s\in\R$, we combine the $L^2-$observability (Theorem \ref{thm3}) with a unique continuation property, which is ensured by the resolvent estimate (Theorem \ref{thm2}). Inspired by \cite{wunsch2015periodic}, we prove these two results by reducing them into a tori setting by the application of \textit{Floquet-Bloch transform}. After a precise analysis of the linear problem, we follow the ideas in \cite{DGL-2006,Laurent-dim3} and decompose the solution $u$ to \eqref{eq: NLS-u-g} into $u=v+\Psi$, where $v$ solves the following equation
\begin{equation*}
\left\{
\begin{array}{l}
\ii\p_t v+\Delta_{x,y} v=-\ep|u|^2u,  \\
v|_{t=T}=0,  
\end{array}
\right.
\end{equation*}
and $\Psi$ is the solution to
\begin{equation*}
\left\{
\begin{array}{l}
\ii\p_t\Psi+\Delta_{x,y} \Psi=\varphi_T\chi_{\Omega}(1-\Delta_{x,y})^{-s}\varphi_T\chi_{\Omega}g,  \\
\Psi|_{t=0}=\Psi_0,\Psi|_{t=T}=0.  
\end{array}
\right.
\end{equation*}
The purpose is then to choose the adequate $\Psi_0$, and the system is completely determined. Let us define $\mathcal{J}: \Psi_0\mapsto u_0-v|_{t=0}$. Thus, seeking the control function $g$ is reduced to finding a fixed point for $\mathcal{J}$, provided that $\|u_0\|_{H^s}$ is small enough (see Section \ref{sec: 4} for details). The general result will follow by
reversing time (presented in Section \ref{sec: 5}).  

\subsection{Organization of the rest of this paper}
In Section \ref{sec: 2}, we prove linear observability and stationary estimates, which includes the proof for Theorem \ref{thm2}; in Section \ref{sec: 3}, we discuss well-posedness theory for NLS which will be used for the nonlinear case; in Section \ref{sec: 4}, we show the local null controllability of NLS, which completes the proof of Theorem \ref{thm4}; in Section \ref{sec: 5}, we obtain the exact controllability, which finishes the proof of Theorem \ref{thm1}.

\subsection{Notations}
Throughout this paper, we use $C$ to denote the universal constant and $C$ may change line by line. We say $A\lesssim B$, if $A\leq CB$. We say $A\sim B$ if $A\lesssim B$ and $B\lesssim A$. We also use the notation $C_{B}$ to denote a constant depends on $B$. We use the standard notation for $L^{p}$ spaces and $L^2$-based Sobolev spaces $H^{s}$.

We define the Fourier transform on $\mathbb{R}^n \times \mathbb{T}^m$ as follows:
\begin{equation}
    (\mathcal{F} f)(\xi)= \int_{\mathbb{R}^n \times \mathbb{T}^m}f(z)e^{-iz\cdot \xi}dz,
\end{equation}
where $\xi=(\xi_1,\xi_2,...,\xi_{d})\in \mathbb{R}^n \times \mathbb{Z}^m$ and $d=m+n$. We also note the Fourier inversion formula
\begin{equation}
    f(z)=c \sum_{(\xi_{n+1},...,\xi_{d})\in \mathbb{Z}^m} \int_{(\xi_1,...,\xi_{n}) \in \mathbb{R}^n} (\mathcal{F} f)(\xi)e^{iz\cdot \xi}d\xi_1...d\xi_n.
\end{equation}
For convenience, we may consider the discrete sum to be the integral with discrete measure so we can combine the above integrals together and treat them to be one integral. Moreover, we define the Schr{\"o}dinger propagator $e^{it\Delta}$ by
\begin{equation}
    \left(\mathcal{F} e^{it\Delta}f\right)(\xi)=e^{-it|\xi|^2}(\mathcal{F} f)(\xi).
\end{equation}

At last, we refer to \cite{fan2024long,ZZ} for the definitions and properties of Fourier restriction spaces (also known as Bourgain spaces) in the waveguide setting. They share very similar properties as the tori case. For the sake of completeness, we include some properties of these function spaces in the Appendix \ref{app}. These spaces are very useful and frequently used for the study of NLS on tori or waveguide manifolds.\footnote{We also refer to \cite{HTT1,HTT2} and the references therein for another type of function spaces, which are also very useful for the study of NLS on tori or waveguide manifolds (especially for the critical case).}
\subsection*{Acknowledgment}
We highly appreciate Prof. N. Burq for helpful suggestions and insightful discussions. J. Niu is supported by Defi Inria EQIP. Z. Zhao is supported by the NSF grant of China (No. 12101046, 12271032, 2426205) and the Beijing Institute of Technology Research Fund Program for Young Scholars.

\section{Linear observability and Stationary estimates}\label{sec: 2}
\subsection{The Floquet-Bloch transform}
In this part, we introduce the partial Floquet-Bloch transform. This tool is instrumental in the proof of Theorem \ref{thm2}. For more details, we refer to \cite{wunsch2015periodic}, \cite[Section 4]{Kuchment} and \cite[Section 2.1]{BM2023}. We first introduce the partial Floquet transform.
\begin{definition}
We define the partial Floquet-Bloch transform $\mathcal{T}^p: L^2(\R^n\times\T^m)\rightarrow L^2(\T^{m+n}\times[0,1)^n)$ by
\begin{equation*}
\mathcal{T}^pu(x,y,\alpha):=e^{\ii x\cdot \alpha}\sum_{k\in \mathbb{Z}^n}e^{2\pi\ii\alpha\cdot k}u(x+2k\pi,y),\;\forall u\in L^2(\R^n\times\T^m),\forall (x,y,\alpha)\in \T^{n}\times\T^{m}\times[0,1)^n. 
\end{equation*}

\begin{remark}
In particular, if $m=0$, we obtain the usual Floquet-Bloch transform. Furthermore, removing the factor $e^{\ii\alpha\cdot x}$, we obtain a partial Floquet transform 
\begin{equation*}
\Pi^pu(x,y,\alpha):=\sum_{k\in \mathbb{Z}^n}e^{2\pi\ii\alpha\cdot k}u(x+2k\pi,y),\;\forall u\in L^2(\R^n\times\T^m),\forall (x,y,\alpha)\in \T^{n}\times\T^{m}\times[0,1)^n. 
\end{equation*}
\end{remark}
\end{definition}
For convenience, for any $\alpha\in[0,1)^n$ and $ u\in L^2(\R^n\times\T^m)$, we define $\Pi_{\alpha}: L^2(\R^n\times\T^m)\rightarrow L^2(\T^{m+n})$ by  
\begin{equation}
    \Pi_{\alpha}u(x,y):=e^{\ii x\cdot \alpha}\sum_{k\in \mathbb{Z}^n}e^{2\pi\ii\alpha\cdot k}u(x+2k\pi,y),\;\forall (x,y)\in \T^{n}\times\T^{m}.
\end{equation}
The next proposition ensures that the partial Floquet-Bloch transform is an isometry from $L^2(\R^n\times\T^m)$ to $L^2(\T^{m+n}\times[0,1)^n)$.
\begin{proposition}
The map $\mathcal{T}^p$ is an isometric isomorphism from $L^2(\R^n\times\T^m)$ to $L^2(\T^{m+n}\times[0,1)^n)$.
\end{proposition}
This proposition is based on the following lemma.
\begin{lemma}\label{lem: L^2 norm equivalence}
We have the equality of $L^2$ norms    
\begin{equation}\label{eq: L^2-norm isometry}
  \|g\|^2_{L^2(\mathbb{R}^n \times \mathbb{T}^m)}=\int_{[0,1)^n} \|\Pi_{\alpha} g\|^2_{L^2(\mathbb{T}^n \times \mathbb{T}^m)} d\alpha  .
\end{equation}
More generally, if $\Omega_1 \in \mathbb{R}^n$ is $2\pi \mathbb{Z}^n$-invariant and $\Omega_0$ denotes its projection
to $\mathbb{T}^n$
\begin{equation}
     \|g\|^2_{L^2(\Omega_1 \times \mathbb{T}^m)}=\int_{[0,1)^n} \|\Pi_{\alpha} g\|^2_{L^2(\Omega_0 \times \mathbb{T}^m)} d\alpha .
\end{equation}
\end{lemma}
\begin{proof}
We use Fubini to compute the Fourier coefficients of the periodic
functions $\Pi_{\alpha}g$ on $\T^{m+n}$.
\begin{align*}
    \widehat{\Pi_{\alpha}g}(k,l)&=\frac{1}{(2\pi)^{\frac{n+m}{2}}}\int_{\T^m}\int_{\T^n}e^{-\ii x\cdot k}e^{-\ii y\cdot l}\sum_{j\in\Z^n}e^{\ii x\cdot \alpha}e^{2\pi\ii\alpha\cdot j}g(x+2j\pi,y)dxdy\\
    &=\frac{1}{(2\pi)^{\frac{n+m}{2}}}\int_{\T^m}\int_{\R^n}e^{\ii x\cdot (\alpha-k)}e^{-\ii y\cdot l}g(x,y)dxdy\\
    &=\mathcal{F}_x (\widehat{g}(\cdot,l))(k-\alpha).
\end{align*}
Integrating the sum of squares of the right hand side over the unit cube $[0,1)^n$ gives $\|g\|^2_{L^2(\mathbb{R}^n \times \mathbb{T}^m)}$ by Fubini and Plancherel theorems on $\R^n\times\T^m$, while on the left hand side, we get
\begin{equation*}
\int_{[0,1)^n}\sum_{k\in\Z^n,l\in\Z^m}| \widehat{\Pi_{\alpha}g}(k,l)|^2d\alpha=\int_{[0,1)^n}\| \Pi_{\alpha}g\|_{L^2(\T^{n+m})}^2d\alpha
\end{equation*}
by Plancherel on the torus. The generalization to taking the norm over $\Omega_1\times\Omega_2$ is proved simply by applying \eqref{eq: L^2-norm isometry} to the function $\mathds{1}_{\Omega_1\times\Omega_2}g$.
\end{proof}

\subsection{Observability for partial-twisted Laplacian}
Let $\alpha\in\R^n$, we set 
\begin{equation*}
    H_{\alpha}:=(\partial_x-\ii \alpha)^2+\Delta_y.
\end{equation*}
Note that these operators are all self-adjoint with the same domain independent of $\alpha$. We first recall the Schr\"odinger observability on the torus
\begin{theorem}{\cite[Theorem 4]{AM-JEMS}}
For every non-empty open set $\Omega\subset\T^d$ and every $T>0$ there exists a constant $C=C(T,\Omega)$ such that 
\begin{equation*}
    \|u_0\|^2_{L^2(\T^d)}\leq C\int_0^T\|e^{\ii t \Delta}u_0\|^2_{L^2(\Omega)}dt
\end{equation*}
for every initial datum $u_0\in L^2(\T^d)$.
\end{theorem}
We use this classical theorem to prove the following lemma
\begin{lemma}\label{lem: ob for H}
For every non-empty open set $\Omega=(\Omega_0,\Omega_2)\subset\T^{n}\times\T^m$, and every $T>0$ there exists a constant $C=C(T,\Omega)$ such that 
\begin{equation*}
    \|u_0\|^2_{L^2(\T^{n+m})}\leq C\int_0^T\|e^{\ii t H_{\alpha}}u_0\|^2_{L^2(\Omega)}dt
\end{equation*}
for every initial datum $u_0\in L^2(\T^{n+m})$.
\end{lemma}
\begin{proof}
We point out that this lemma is almost a restatement of \cite[Proposition 4]{wunsch2015periodic}. When $m=0$, the twisted Laplacian is defined in \cite[Eq. (3)]{wunsch2015periodic} satisfies a similar observability. The proof of \cite[Proposition 4]{wunsch2015periodic} can also be applied here without changes. For the completeness of our paper, we present the proof adapted to our setting.
For $\Omega_0\subset\T^{n}$, we can find a nonempty open set $\omega_0\subset\T^n$ such that $\overline{\omega_0}\subset \Omega_0$. We set $\omega=(\omega_0,\Omega_2)\subset\Omega$. Hence there exists $T > 0$ such that $x_1\in\omega_0$ and $d(x_1, x_2) < 4\sqrt{n}T$ imply $x_2\in\Omega_0$. Using the operator identity
\begin{equation*}
e^{\ii tH_{\alpha}}=e^{\ii t(\partial_x-\ii\alpha)^2}e^{\ii\Delta_y}=e^{-\ii t|\alpha|^2}\tau_{(2t\alpha,0)}e^{\ii t\Delta_{x,y}},
\end{equation*}
where for $\theta=(\theta_1,\theta_2)\in \R^{n+m}$, $\tau_{\theta}$ denotes the translation operator $\tau_{\theta}f(x,y):=f(x+\theta_1,y+\theta_2)$. Thus, by the $H_0-$observability and the choice of $T\ll 1$ so that $\tau_{(-2t\alpha,0)}(\omega)\subset \Omega$ for $\forall t\in[0,T]$, we obtain 
\begin{align*}
\|u_0\|^2_{L^2(\T^{n+m})}&\leq C\int_0^T\|e^{\ii t \Delta_{x,y}}u_0\|^2_{L^2(\omega)}dt\\
&\leq C\int_0^T\|e^{-\ii t|\alpha|^2}\tau_{(2t\alpha,0)}e^{\ii t\Delta_{x,y}}u_0\|^2_{L^2(\tau_{(-2t\alpha,0)}(\omega))}dt\\
&\leq C\int_0^T\|e^{-\ii t|\alpha|^2}\tau_{(2t\alpha,0)}e^{\ii t\Delta_{x,y}}u_0\|^2_{L^2(\Omega)}dt\\
&\leq C\int_0^T\|e^{\ii t H_{\alpha}}u_0\|^2_{L^2(\Omega)}dt.
\end{align*}
\end{proof}
Thanks to \cite[Theorem 5.1]{Miller-JFA}, we have the following proposition.
\begin{proposition}\label{prop: resolvent on torus}
Let $\Omega=(\Omega_0,\Omega_2) \subset \mathbb{T}^n\times\T^m$ be open and nonempty. For all $\alpha \in [0,1)^n$ and
\begin{equation}
(H_{\alpha}-\lambda)u=f,    
\end{equation}
posed on $\mathbb{T}^{n+m}$, we have,
\begin{equation}
\|u\|_{L^2(\mathbb{T}^{n+m})} \leq C(\|f\|_{L^2( \mathbb{T}^{n+m})}+\|u\|_{L^2(\Omega)}),
\end{equation}
with constants independent of $\alpha$ and $\lambda \in \mathbb{R}$.
\end{proposition}
\subsection{Stationary estimate: proof of Theorem \ref{thm2}}
\begin{proof}[Proof of Theorem \ref{thm2}]
First, we notice that
\begin{equation}
    H_{\alpha}-\lambda=(\partial_{x}-\ii\alpha)^2+\Delta_y-\lambda=e^{\ii\alpha \cdot x}(\Delta_{x,y}-\lambda)e^{-i\alpha \cdot x},
\end{equation}
where $\alpha\in [0,1)^n$. Thus $(\Delta_{x,y}-\lambda)u=f$ yields
\begin{equation}
    (H_{\alpha}-\lambda)e^{\ii\alpha\cdot x}u=e^{\ii\alpha\cdot x}f.
\end{equation}
Applying the partial Floquet trnaform $\Pi^p$ on both sides  and using translation-invariance of $H_{\alpha}$, we get an equation on the torus:
\begin{equation}
   (H_{\alpha}-\lambda)\Pi_{\alpha} u=\Pi_{\alpha} f \textmd{ on } \mathbb{T}^{n+m}.    
\end{equation}
Applying Proposition \ref{prop: resolvent on torus}, we obtain for every $\alpha$ in a fundamental domain
(and with constants independent of $\alpha$)
\begin{equation*}
\|\Pi_{\alpha} u\|_{L^2(\mathbb{T}^{n+m})} \leq C(\|\Pi_{\alpha} f\|_{L^2( \mathbb{T}^{n+m})}+\|\Pi_{\alpha} u\|_{L^2(\Omega_0)}).
\end{equation*}
Now by Lemma \ref{lem: L^2 norm equivalence} we may integrate both sides in $\alpha_1$ to obtain Theorem \ref{thm2}.
\end{proof}

\begin{remark}
The idea of establishing the estimate \eqref{eq: est-unique-continuation} in a periodic setting through a twisted Laplacian was first raised by Wunsch in \cite{wunsch2015periodic}. The correspondence between $H_{\alpha}$ on tori and $\Delta$ in periodic setting introduced in \cite{wunsch2015periodic} may look simple but turns out to be a powerful tool. 
\end{remark}

\subsection{Observability in $\R^m\times\T^n$: Proof of Theorem \ref{thm3}}
\begin{proof}[Proof of Theorem \ref{thm3}]
Let $u(t):=e^{\ii t\Delta_{x,y}}u_0$ and $v(t):=\Pi_{\alpha}u(t)$, for $\alpha\in[0,1)^n$. Then $u$ is the solution to 
\begin{equation*}
\lf\{
\begin{array}{ll}
     (\ii\partial_t +\Delta_{x,y})u=0&\text{in } (0,T)\times\R^n\times\T^m,  \\
     u|_{t=0}=u_0& \text{in } \R^n\times\T^m,
\end{array}
\ri.
\end{equation*}
and $v$ is the solution to 
\begin{equation*}
\lf\{
\begin{array}{ll}
     (\ii\partial_t +H_{\alpha})v=0&\text{in } (0,T)\times\T^n\times\T^m, \\
     v|_{t=0}=\Pi_{\alpha}u_0& \text{in } \T^n\times\T^m.
\end{array}
\ri.
\end{equation*}
By Lemma \ref{lem: ob for H}, we obtain
\begin{equation*}
\|\Pi_{\alpha}u_0\|^2_{L^2(\T^n\times\T^m)}\leq C\int_0^T\|v(t)\|^2_{L^2(\Omega)}dt=C\int_0^T\|\Pi_{\alpha}u(t)\|^2_{L^2(\Omega)}dt.
\end{equation*}
Integrating the LHS over $[0,1)^n$ gives $\|u_0\|^2_{L^2(\R^n\times\T^m)}$ by Lemma \ref{lem: L^2 norm equivalence}, while on the RHS, we get
\begin{equation*}
\int_{[0,1)^n}\int_0^T\|\Pi_{\alpha}u(t)\|^2_{L^2(\Omega)}dtd\alpha=\int_0^T\|u(t)\|^2_{L^2(\Omega)}dt,
\end{equation*}
by Fubini and Lemma \ref{lem: L^2 norm equivalence}. We conclude that
\begin{equation*}
\|u_0\|^2_{L^2(\R^n\times\T^m)}\leq C\int_0^T\|e^{\ii t\Delta_{x,y}}u_0\|^2_{L^2(\Omega)}dt
\end{equation*}
\end{proof}
\begin{proposition}\label{prop: H^s-ob}
Let $\Omega=(\Omega_1,\Omega_2)$ satisfy the condition (\textbf{G}). For every $T>0$ and $s\geq1$, there exists a constant $C=C(T,\Omega,s,\chi_{\Omega})$ such that for $\forall u_0\in H^s(\mathbb{R}^n \times \mathbb{T}^m)$
\begin{equation}\label{eq: H^s-ob-chi}
\|u_0\|^2_{H^s(\mathbb{R}^n \times \mathbb{T}^m)}\leq C\int_0^T\|\chi_{\Omega}e^{\ii t\Delta_{x,y}}u_0\|^2_{H^s(\mathbb{R}^n \times \mathbb{T}^m)}dt,
\end{equation}
where $0\leq \chi_{\Omega}\in C^{\infty}(\mathbb{R}^n \times \mathbb{T}^m)$ satisfies $\mathbf{1}_{\Omega}\leq\chi_{\Omega}$.    
\end{proposition}
\begin{proof}
By the definition of $\chi_{\Omega}$ and Theorem \ref{thm3}, we can easily derive that $\exists C_0=C_0(T,\Omega)>0$ such that for $\forall u_0\in L^2(\mathbb{R}^n \times \mathbb{T}^m)$ 
\begin{equation}\label{eq: L^2-ob-chi}
\|u_0\|^2_{L^2(\R^n\times\T^m)}\leq C_0\int_0^T\|\chi_{\Omega}e^{\ii t\Delta_{x,y}}u_0\|^2_{L^2(\mathbb{R}^n \times \mathbb{T}^m)}dt.
\end{equation}
For the initial datum $u_0\in  H^s(\mathbb{R}^n \times \mathbb{T}^m)$, 
\begin{equation*}
\|u_0\|^2_{H^s(\mathbb{R}^n \times \mathbb{T}^m)}=\|(1-\Delta_{x,y})^{\frac{s}{2}}u_0\|^2_{L^2(\mathbb{R}^n \times \mathbb{T}^m)}\leq C_0\int_0^T\|\chi_{\Omega}e^{\ii t\Delta_{x,y}}(1-\Delta_{x,y})^{\frac{s}{2}}u_0\|^2_{L^2(\mathbb{R}^n \times \mathbb{T}^m)}dt.    
\end{equation*}
Thanks to Cauchy–Schwarz inequality and the fact that
$$
\chi_{\Omega}e^{\ii t\Delta_{x,y}}(1-\Delta_{x,y})^{\frac{s}{2}}u_0=(1-\Delta_{x,y})^{\frac{s}{2}}\chi_{\Omega}e^{\ii t\Delta_{x,y}}u_0+[\chi_{\Omega},(1-\Delta_{x,y})^{\frac{s}{2}}]e^{\ii t\Delta_{x,y}}u_0,
$$
we obtain
\begin{align*}
\|u_0\|^2_{H^s(\mathbb{R}^n \times \mathbb{T}^m)}&\leq 2C_0\int_0^T\|(1-\Delta_{x,y})^{\frac{s}{2}}\chi_{\Omega}e^{\ii t\Delta_{x,y}}u_0\|^2_{L^2(\mathbb{R}^n \times \mathbb{T}^m)}dt\\
&+2C_0\int_0^T\|[\chi_{\Omega},(1-\Delta_{x,y})^{\frac{s}{2}}]e^{\ii t\Delta_{x,y}}u_0\|^2_{L^2(\mathbb{R}^n \times \mathbb{T}^m)}dt.
\end{align*}
Since the commutator $[\chi_{\Omega},(1-\Delta_{x,y})^{\frac{s}{2}}]$ is a pseudo-differential operator of order $s-1$, we know that 
$\|[\chi_{\Omega},(1-\Delta_{x,y})^{\frac{s}{2}}]e^{\ii t\Delta_{x,y}}u_0\|^2_{L^2(\mathbb{R}^n \times \mathbb{T}^m)}\leq C\|e^{\ii t\Delta_{x,y}}u_0\|^2_{H^{s-1}(\mathbb{R}^n \times \mathbb{T}^m)}\leq C\|u_0\|^2_{H^{s-1}(\mathbb{R}^n \times \mathbb{T}^m)}$. 
Thus,
\begin{equation}\label{eq: weak-ob-H^s}
\|u_0\|^2_{H^s(\mathbb{R}^n \times \mathbb{T}^m)}\leq 2C_0\int_0^T\|\chi_{\Omega}e^{\ii t\Delta_{x,y}}u_0\|^2_{H^s(\mathbb{R}^n \times \mathbb{T}^m)}dt+2C_0CT\|u_0\|^2_{H^{s-1}(\mathbb{R}^n \times \mathbb{T}^m)}.
\end{equation}
This is the so-called \textit{weak observability} in $H^s(\mathbb{R}^n \times \mathbb{T}^m)$. We apply a standard procedure based on compactness-uniqueness arguments to derive the observability \eqref{eq: H^s-ob-chi} from the weak observability \eqref{eq: weak-ob-H^s}. Suppose that \eqref{eq: H^s-ob-chi} is false. 
Then we find a sequence $\{u_0^k\}_{k\in\N^*}\subset H^s(\mathbb{R}^n \times \mathbb{T}^m)$ such that 
\begin{gather}
\|u_0^k\|^2_{H^s(\mathbb{R}^n \times \mathbb{T}^m)}=1,\label{eq: norm-1}\\
\int_0^{T}\|\chi_{\Omega}e^{\ii t\Delta_{x,y}}u_0^k\|^2_{H^s(\mathbb{R}^n \times \mathbb{T}^m)}dt\rightarrow0.\label{eq: limit of integral}
\end{gather}
The second statement \eqref{eq: limit of integral} leads to 
\begin{equation}\label{eq: vanish in subdomain}
\chi_{\Omega}e^{\ii t\Delta_{x,y}}u_0^k\rightarrow0, \text{ strongly in }L^2((0,T);H^s(\mathbb{R}^n \times \mathbb{T}^m)), \text{ as }k\rightarrow\infty.
\end{equation}
Thanks to \eqref{eq: norm-1}, we can extract a subsequence (still denoted by $u_0^k$ for simplicity) such that $u_0^k\rightharpoonup u$ weakly in $H^s(\mathbb{R}^n \times \mathbb{T}^m)$. By the compact inclusion $i: H^s(\mathbb{R}^n \times \mathbb{T}^m)\rightarrow H^{s-1}(\mathbb{R}^n \times \mathbb{T}^m)$, we deduce that $u_0^k\rightarrow u$ strongly in $H^{s-1}(\mathbb{R}^n \times \mathbb{T}^m)$. Then we define a subspace $\mathcal{N}_{T_0}$ in $H^s(\mathbb{R}^n \times \mathbb{T}^m)$, collecting initial data that generate invisible solutions in $(0,T_0)$, by
\begin{equation}
\mathcal{N}_{T_0}:=\{f\in H^s(\mathbb{R}^n \times \mathbb{T}^m): \chi_{\Omega}e^{\ii t\Delta_{x,y}}f=0 \text{ for }t\in(0,T_0)\}.
\end{equation}
\begin{lemma}\label{lem: null space}
$\mathcal{N}_{T}=\{0\}$.
\end{lemma}
Using Lemma \ref{lem: null space}, we finish the proof of the observability. Using the weak observability \eqref{eq: weak-ob-H^s} and the condition \eqref{eq: limit of integral}, we obtain
\begin{align*}
1=\lim_{k\rightarrow\infty}\|u_0^k\|^2_{H^s(\mathbb{R}^n \times \mathbb{T}^m)}&\leq \lim_{k\rightarrow\infty}\left(2C_0\int_0^{T}\|\chi_{\Omega}e^{\ii t\Delta_{x,y}}u^k_0\|^2_{H^s(\mathbb{R}^n \times \mathbb{T}^m)}dt+2C_0CT\|u^k_0\|^2_{H^{s-1}(\mathbb{R}^n \times \mathbb{T}^m)}\right)\\
&\leq \lim_{k\rightarrow\infty}2C_0CT\|u^k_0\|^2_{H^{s-1}(\mathbb{R}^n \times \mathbb{T}^m)}\\
&\leq 2C_0CT\lim_{k\rightarrow\infty}\|u^k_0\|^2_{H^{s-1}(\mathbb{R}^n \times \mathbb{T}^m)}.
\end{align*}
By the strong convergence of $u^k_0$ in $H^{s-1}(\mathbb{R}^n \times \mathbb{T}^m)$ and the condition \eqref{eq: vanish in subdomain}, we know that $u\in \mathcal{N}_{T}=\{0\}$ and $1\leq 2C_0CT\|u\|^2_{H^{s-1}(\mathbb{R}^n \times \mathbb{T}^m)}$, which is a contradiction. Hence, the observability \eqref{eq: H^s-ob-chi} holds.

\end{proof}
Now we are in a position to prove Lemma \ref{lem: null space}.
\begin{proof}[Proof of Lemma \ref{lem: null space}]
By the definition of $\mathcal{N}_T$, for any $\delta>0$, we deduce that $\mathcal{N}_{T+\delta}\subset\mathcal{N}_{T}$. Then, we claim that $\mathcal{N}_{T}$ is of finite dimension $\forall T>0$. Indeed, thanks to the weak observability and the definition of $\mathcal{N}_{T}$, $\forall f\in \mathcal{N}_{T}$,
\begin{align*}
\|f\|^2_{H^s(\mathbb{R}^n \times \mathbb{T}^m)}&\leq 2C_0\int_0^{T}\|\chi_{\Omega}e^{\ii t\Delta_{x,y}}f\|^2_{H^s(\mathbb{R}^n \times \mathbb{T}^m)}dt+2C_0CT\|f\|^2_{H^{s-1}(\mathbb{R}^n \times \mathbb{T}^m)}\\
&= 2C_0CT\|f\|^2_{H^{s-1}(\mathbb{R}^n \times \mathbb{T}^m)}.
\end{align*}
This implies that $\mathcal{N}_{T}$ is of finite dimension. Moreover, combining with the fact that $\mathcal{N}_{T_0+\delta}\subset\mathcal{N}_{T_0}$, $\forall \delta>0$, there exists $\delta_0>0$ such that $\forall \delta\in(0,\delta_0)$, $\mathcal{N}_{T_0+\delta}=\mathcal{N}_{T_0}$. 
For any $\delta\in(0,\delta_0)$ and $\forall f\in \mathcal{N}_{T_0+\delta}=\mathcal{N}_{T_0}$, $\chi_{\Omega}e^{\ii t\Delta_{x,y}}e^{\ii\delta\Delta_{x,y}}f=0$, for $\forall t\in(0,T_0)$. Consequently, we know that $e^{\ii\delta\Delta_{x,y}}f\in\mathcal{N}_{T_0}$. Thus, $\Delta_{x,y}f=\lim_{\delta\rightarrow0^+}\frac{e^{\ii\delta\Delta_{x,y}}f-f}{\ii\delta}\in\mathcal{N}_{T_0}$, which implies that $\mathcal{N}_{T_0}$ is stable under $\Delta_{x,y}$. Suppose $\mathcal{N}_{T_0}\neq\{0\}$. Since $\mathcal{N}_{T_0}$ is of finite dimension, it must contain an eigenfunction of $\Delta_{x,y}$. Let us denote this eigenfunction by $\phi\neq0$ and its associated eigenvalue $\lambda_0\in\R$. Thus, $\phi$ satisfies
\begin{equation*}
    \Delta_{x,y}\phi=\lambda_0\phi,\chi_{\Omega}e^{\ii \lambda_0t}\phi=0,\forall t\in(0,T_0).
\end{equation*}
Applying Theorem \ref{thm3}, we know that $\phi\equiv0$, which is a contradiction. This leads to $\mathcal{N}_{T_0}=\{0\}$.
\end{proof}

\section{Well-posedness of NLS}\label{sec: 3}
In this section, we include some well-posedness results of NLS. As in Theorem \ref{thm1} and Theorem \ref{thm4}, we consider a specific and popular model: 3D cubic NLS in energy space and higher.
\begin{theorem}\label{thm: well-posedness of NLS-source}
Let $T>0$, $s\geq 1$, and $\ep=\pm1$. For every $f\in L^2([-T,T],H^s)$ and $u_0\in H^s$, there is a unique solution $u\in X^{s,b}_T$ to
\begin{equation}\label{eq: nls-f}
\left\{
\begin{array}{ll}
\ii\p_t u+\Delta u+\ep|u|^{2}u=f &\text{ on }[-T,T]\times\R^2\times\T  \\
 u|_{t=0}=u_0    & \text{ on }\R^2\times\T
\end{array}
\right.
\end{equation}
Moreover the flow map is Lipschitz on every bounded subset.
\end{theorem}
\begin{remark}
We can replace $\R^2\times\T$ by $\R\times\T^2$ (still 3-dimensional manifold) and the proof follows in the same way since the corresponding estimates still work.    
\end{remark}
\begin{proof}\label{Sketch of the proof}
The proof is similar to the proof of Theorem 2.1 in \cite{laurent2010global}. In fact, our case is easier since we do not have a damping term. We only need to handle the source term $f$.

First, we notice that if $f\in L^2([-T,T],H^s)$, it also belongs to $X^{s,-b'}_T$ as $b'\geq 0$ due to the property of function spaces. Moreover, it suffices to consider positive times. The solution on $[-T,0]$ can be obtained similarly.

We consider the functional
 $$\Phi(u)(t)=e^{it\Delta_x}u_0-i\int_0^t e^{i(t-\tau)\Delta_x}\left[\lambda \left|u\right|^2 u+ f\right](\tau) d\tau .$$
One can apply a standard fixed point argument on the Banach space $X^{s,b}_T$.

We recall that, for $T\leq 1$ we have
$$\left\|e^{it\Delta_x}u_0\right\|_{X^{s,b}_T}\leq C \left\|u_0\right\|_{H^s}.$$

This handles the linear term. Next, regarding the source term, one can use
$ \left\|\psi(t/T)\int e^{i(t-\tau)\Delta_x}F(\tau)\right\|_{X^{s,b}} \leq CT^{1-b-b'}\left\|F\right\|_{X^{s,-b'}}$ to deal with the source term $f$.

Then the proof can be reduced to the case without the source. We refer to Bourgain \cite{bourgain1993fourier}.
\end{proof}
For the nonlinear equation \eqref{eq: nls-f}, we can decompose it near a linear solution $\Psi$ given by
\begin{equation}
\left\{
\begin{array}{l}
\ii\p_t\Psi+\Delta_{x,y} \Psi=f,  \\
 \Psi|_{t=0}=\Psi_0\in H^s(\R^2\times\T)  
\end{array}
\right.
\end{equation}
We write $u=\Psi+v$. Then $v$ satisfies 
\begin{equation}\label{eq: v-NLS-well-posedness}
\left\{
\begin{array}{l}
\ii\p_t v+\Delta_{x,y} v+\ep\left(\left|\Psi\right|^2\Psi+2\left|\Psi\right|^2v+\Psi^2\bar{v}\right)=-\ep F(\Psi,v),  \\
 v|_{t=T}=0. 
\end{array}
\right.
\end{equation}
where $F(\Psi,v):=\left|v\right|^2v+2\left|v\right|^2\Psi+v^2\bar{\Psi}$. 
In the following proposition, we present the well-posedness of $v$.
\begin{proposition}\label{prop: wellposedness-v}
Let $T>0$, $s\geq 1$, and $\ep=\pm1$. For every $\Psi_0\in H^s$, there is a unique solution $v\in X^{s,b}_T$ to the equation \eqref{eq: v-NLS-well-posedness}.
\end{proposition}
\begin{proof}
Since the proof is very similar to Theorem \ref{thm: well-posedness of NLS-source}, we only give a sketch for it. 

First, we note that, for any $\Psi_0\in H^s$, the solution $\Psi$ to the linear equation \eqref{eq: Psi-LS} satisfies $\Psi\in  X^{s,b}_T$. This observation allows us to control for $\Psi$. 

Then, compared to Theorem \ref{thm: well-posedness of NLS-source}, there is no source term (that is good for us) in \eqref{eq: v-NLS} and there are more nonlinear terms. We note that all nonlinear terms are cubic and we can estimate them using Strichartz estimate in a standard way. The proof can be done as in Bourgain \cite{bourgain1993fourier} since we already have control for $\Psi$. 
\end{proof}

\section{Local null controllability of NLS: Proof of Theorem \ref{thm4}}\label{sec: 4}

\subsection{Control for linear problem: Hilbert uniqueness method}\label{sec: HUM}
Consider a linear Schr\"odinger equation
\begin{equation}
\left\{
\begin{array}{ll}
\ii\p_t u+\Delta u=\varphi_T\chi_{\Omega}(1-\Delta_{x,y})^{-s}\varphi_T\chi_{\Omega}f &\text{ on }\R\times\R^2\times\T,  \\
 u|_{t=0}=u_0    & \text{ on }\R^2\times\T,
\end{array}
\right.
\end{equation}
such that $u(T)\equiv0$. Define the range operator $R$ by 
\begin{align*}
R: H^{-s}(\R\times\R^2\times\T)&\rightarrow H^s(\R^2\times\T)\\
f&\mapsto u_0.
\end{align*}
Consider the adjoint system  
\begin{equation}
\left\{
\begin{array}{ll}
\ii\p_t w+\Delta w=0 &\text{ on }\R\times\R^2\times\T,  \\
 w|_{t=0}=w_0    & \text{ on }\R^2\times\T,
\end{array}
\right.
\end{equation}
Define the solution operator $S$ by 
\begin{align*}
S: H^s(\R^2\times\T)&\rightarrow C^0(\R;L^2(\R^2\times\T))\subset H^s(\R\times\R^2\times\T)\\
w_0&\mapsto e^{\ii t\Delta}w_0.
\end{align*}
\begin{proposition}
Let $s\in\R$. Let $f\in H^{-s}(\R\times\R^2\times\T)$ and $w_0\in H^{-s}(\R^2\times\T)$. Then the duality holds as follows:
\begin{equation}
    \psc{\varphi_T\chi_{\Omega}(1-\Delta_{x,y})^{-s}\varphi_T\chi_{\Omega}f}{Sw_0}_{H^{s}(\R^2\times\T),H^{-s}(\R^2\times\T)}=\psc{
    \ii R(f)
    }{w_0}_{H^{s}(\R^2\times\T),H^{-s}(\R^2\times\T)}.
\end{equation}
\end{proposition}
\begin{proof}
We only consider $f$ and $w_0$ are smooth in the space variables. Then the general case is proved by standard approximation arguments. 
\begin{align*}
    \psc{\varphi_T\chi_{\Omega}(1-\Delta_{x,y})^{-s}\varphi_T\chi_{\Omega}f}{Sw_0}_{H^{s},H^{-s}}&=\psc{\ii\p_t u+\Delta u}{e^{\ii t\Delta}w_0}_{H^{s},H^{-s}}\\
    &=\psc{u}{(\ii\p_t+\Delta)e^{\ii t\Delta}w_0}_{H^{s},H^{-s}}+\psc{
    \ii u_0
    }{w_0}_{H^{s},H^{-s}}\\
    &=\psc{
    \ii R(f)
    }{w_0}_{H^{s},H^{-s}}.
\end{align*}
\end{proof}
 Then we can define the HUM operator $\mathcal{K}=\ii R\circ S$, which satisfies the following proposition.
 \begin{proposition}\label{prop: iso-HUM-op}
$\mathcal{K}$ is an isomorphism from $H^{-s}$ to $H^s$.    
\end{proposition}
\begin{proof}
We define a continuous form by \(\alpha(u,v):=\psc{\mathcal{K}u}{v}_{H^{s}(\R^2\times\T),H^{-s}(\R^2\times\T)}\), for $u,v\in H^{-s}(\R^2\times\T)$. By the definition of $\mathcal{K}$, it is easy to verify that $\mathcal{K}$ is self-adjoint operator. Then we check the coercivity of $\alpha$. Thanks to the observability estimate in Theorem \ref{thm3}, for $\forall v\in L^2(\R^2\times\T)$,
\begin{align*}
\alpha(v,v)=\psc{\mathcal{K}v}{v}_{H^{s},H^{-s}}&=\ps{(1-\Delta_{x,y})^{-\frac{s}{2}}\varphi_T\chi_{\Omega}Sv}{(1-\Delta_{x,y})^{-\frac{s}{2}}\varphi_T\chi_{\Omega}Sv}_{L^2(\R^2\times\T)} \\
&=\|\varphi_T\chi_{\Omega}Sv\|^2_{H^{-s}(\R^2\times\T)}.
\end{align*}
By Proposition \ref{prop: H^s-ob}, there exists a constant $C=C(T,\Omega,s,\chi_{\Omega})>0$ such that
\[
\|\varphi_T\chi_{\Omega}Sv\|^2_{H^{-s}(\R^2\times\T)}\geq C\|v\|^2_{H^{-s}(\R^2\times\T)}.
\]
By Lax-Milgram's theorem, 
% for $\forall u_0\in L^2(\R^2\times\T)$, there exists a unique $u_{R}\in L^2(\R^2\times\T)$ such that \(\alpha(u_{R},V)=\ps{u_0}{v}_{L^2(\R^2\times\T)},\forall v\in L^2(\R^2\times\T)\).Therefore, 
we prove that the HUM operator $\mathcal{K}$ is an isomorphism $H^{-s}$ to $H^s$.
\end{proof}
Define the control operator $\mathcal{L}$ by $\mathcal{L}=\ii S\circ \mathcal{K}^{-1}:H^{s}\rightarrow H^{-s}$. Since $R\circ\mathcal{L}\Psi_0=\Psi_0$, the solution $\Psi$ to the linear Schr\"odinger equation with an initial datum $\Psi_0$
\begin{equation}\label{eq: Psi-LS}
\left\{
\begin{array}{l}
\ii\p_t\Psi+\Delta_{x,y} \Psi=\varphi_T\chi_{\Omega}(1-\Delta_{x,y})^{-s}\varphi_T\chi_{\Omega}\mathcal{L}\Psi_0,  \\
 \Psi|_{t=0}=\Psi_0\in H^s(\R^2\times\T)  
\end{array}
\right.
\end{equation}
satisfies that $\Psi(T)=0$. 

\subsection{Proof of Theorem \ref{thm4}}
Now we consider the null control problem for NLS. Recall that we seek a control function $g$ such that the solution $u$ to
\begin{equation}\label{eq: u-general-nls}
\left\{
\begin{array}{ll}
\ii\p_t u+\Delta_{x,y} u+\ep|u|^{2}u=\varphi_T\chi_{\Omega}(1-\Delta_{x,y})^{-s}\varphi_T\chi_{\Omega}g&\text{ on }[0,T]\times\R^2\times\T,  \\
 u|_{t=0}=u_0    & \text{ on }\R^2\times\T,
\end{array}
\right.
\end{equation}
such that $u|_{t=T}=0$. 
\begin{proof}[Proof of Theorem \ref{thm4}]
First, we decompose $u=v+\Psi$, where $\Psi$ is defined by \eqref{eq: Psi-LS}. Then, $v$ is the solution to
\begin{equation}\label{eq: v-NLS}
\left\{
\begin{array}{l}
\ii\p_t v+\Delta_{x,y} v+\ep\left(\left|\Psi\right|^2\Psi+2\left|\Psi\right|^2v+\Psi^2\bar{v}\right)=-\ep F(\Psi,v),  \\
 v|_{t=T}=0. 
\end{array}
\right.
\end{equation}
where $F(\Psi,v):=\left|v\right|^2v+2\left|v\right|^2\Psi+v^2\bar{\Psi}$. Thus, $\ii\p_t v+\Delta_{x,y} v+\ep\left|\Psi+v\right|^2(\Psi+v)=0$. 
Thanks to Proposition \ref{prop: wellposedness-v}, we know that $v\in X^{s,b}_T$. Let us define a nonlinear solution operator $S^v_N: H^{s}\rightarrow H^s$ associated with the equation \eqref{eq: v-NLS} by $S^v_N\Psi_{0}=v|_{t=0}$,
where $v$ solves the equation \eqref{eq: v-NLS} and $\Psi$ solves the equation \eqref{eq: Psi-LS} with an initial datum $\Psi_0$.
Due to $u=v+\Psi$, $u$ satisfies 
\begin{equation}\label{eq: u-NLS}
\left\{
\begin{array}{ll}
\ii\p_t u+\Delta_{x,y} u+\ep|u|^{2}u=\varphi_T\chi_{\Omega}(1-\Delta_{x,y})^{-s}\varphi_T\chi_{\Omega}\mathcal{L}\Psi_0,\\
 u|_{t=T}=v|_{t=T}+\Psi|_{t=T}=0.
\end{array}
\right.
\end{equation}
As a consequence, for any $\Psi_0\in H^s$, the control function $g=\mathcal{L}\Psi_0$ steers the initial state $\Psi|_{t=0}+v|_{t=0}$ to $0$.

Define another nonlinear operator $\mathcal{J}: H^{s}\rightarrow H^s$ by
$\mathcal{J}\vartheta=u_0-S^v_N\vartheta$. If $\mathcal{J}$ has a fixed point $\Psi_0$, i.e., $\mathcal{J}\Psi_0=\Psi_0$, then, this $\Psi_0$ produces a function $g=\mathcal{L}\Psi_0$, which achieves the null controllability, due to $u_0=\Psi|_{t=0}+v|_{t=0}=\Psi_0+S_N^v\Psi_0$. Hence, the null controllability of nonlinear Schr\"odinger equation \eqref{eq: NLS-u-g} is reduced to the following lemma.
\begin{lemma}\label{lem: fixed point}
There exists a constant $\eta>0$ such that $\mathcal{J}$ has a fixed point in a ball $B_{H^s}(0,\eta)$ in $H^s$.
\end{lemma}
\end{proof}
We complete the proof of Theorem \ref{thm4} by demonstrating Lemma \ref{lem: fixed point}.
\begin{proof}[Proof of Lemma \ref{lem: fixed point}]
We only need to prove that $\mathcal{J}$ is contracting on a small ball $B_{H^s}(0,\eta)$ in $H^s$, provided with $\|u_0\|_{H^s}$ small enough. In this proof, for simplicity, we can assume that $T<1$. Note that our constant $C$ which may depend on $T$ and $s$ could vary from line to line. 

For any $\Psi_0\in B_{H^s}(0,\eta)$, we first show that $\mathcal{J}\Psi_0\in B_{H^s}(0,\eta)$ for $\eta$ sufficiently small. By the definition of the map $\mathcal{J}$, $\|\mathcal{J}\Psi_0\|_{H^s}\leq \|u|_{t=0}\|_{H^s}+\|v|_{t=0}\|_{H^s}$. 
Thanks to Proposition \ref{prop: wellposedness-v},
\begin{equation*}
\|v\|_{X^{s,b}_T}\leq \||u|^2u\|_{X^{s,-b'}_T}\leq \|u\|^3_{X^{s,b'}_T}.
\end{equation*}
Applying Theorem \ref{thm: well-posedness of NLS-source} and Proposition \ref{prop: iso-HUM-op}, we obtain
\begin{equation*}
\|u\|_{X^{s,b'}_T}\leq C\|\mathcal{L}\Psi_0\|_{H^{-s}}\leq C\|\Psi_0\|_{H^{s}}<C\eta.
\end{equation*}
As a consequence, we obtain
\begin{equation*}
\|v|_{t=0}\|_{H^s}\leq \|v\|_{X^{s,b}_T}\leq\|u\|^3_{X^{s,b'}_T}< C\eta^3.
\end{equation*}
Thus, $\|\mathcal{J}\Psi_0\|_{H^s}\leq \|u_0\|_{H^s}+\|v|_{t=0}\|_{H^s}\leq \|u_0\|_{H^s}+C\eta^3$. Choosing $\eta$ such that $C\eta^2<\frac{1}{2}$ and $\|u_0\|_{H^s}<\frac{\eta}{2}$, we obtain $\|\mathcal{J}\Psi_0\|_{H^s}<\eta$ and $\mathcal{J}$ reproduces the ball $B_{H^s}(0,\eta)$. 

Let us prove that $\mathcal{J}$ is contracting for $H^s-$norm. For any $\vartheta,\Tilde{\vartheta}\in B_{H^s}(0,\eta)$, let $u$ and $\Tilde{u}$ be the solutions to \eqref{eq: u-general-nls} with $g=\mathcal{L}\vartheta$ and $g=\mathcal{L}\Tilde{\vartheta}$, respectively. Similarly, we denote by $v$ and $\Tilde{v}$. Then, $u-\Tilde{u}$ solves the equation:
\begin{equation*}
\left\{
\begin{array}{ll}
\ii\p_t(u-\Tilde{u})+\Delta_{x,y} (u-\Tilde{u})+\ep|u|^{2}u-\ep|\Tilde{u}|^2\Tilde{u}=\varphi_T\chi_{\Omega}(1-\Delta_{x,y})^{-s}\varphi_T\chi_{\Omega}\mathcal{L}(\vartheta-\Tilde{\vartheta}),\\
 (u-\Tilde{u})|_{t=0}=(u-\Tilde{u})|_{t=T}=0.
\end{array}
\right.
\end{equation*}
And $v-\Tilde{v}$ solves the equation:
\begin{equation*}
\left\{
\begin{array}{ll}
\ii\p_t(v-\Tilde{v})+\Delta_{x,y} (v-\Tilde{v})+\ep|u|^{2}u-\ep|\Tilde{u}|^2\Tilde{u}=0,\\
 (v-\Tilde{v})|_{t=T}=0.
\end{array}
\right.
\end{equation*}
As we presented in the proof of Theorem \ref{thm: well-posedness of NLS-source}, we have
\begin{align*}
\|u-\Tilde{u}\|_{X^{s,b}_T}&\leq C\||u|^{2}u-|\Tilde{u}|^2\Tilde{u}\|_{X^{s,-b'}_T}+C\|\mathcal{L}(\vartheta-\Tilde{\vartheta})\|_{H^{-s}}\\
&\leq C(\|u\|^2_{X^{s,b'}_T}+\|\Tilde{u}\|^2_{X^{s,b'}_T})\|u-\Tilde{u}\|_{X^{s,b'}_T}+C\|\vartheta-\Tilde{\vartheta}\|_{H^{s}}\\
&\leq  C\eta^2\|u-\Tilde{u}\|_{X^{s,b'}_T}+C\|\vartheta-\Tilde{\vartheta}\|_{H^{s}}.
\end{align*}
For $\eta$ sufficiently small, it yields $C\eta^2<\frac{1}{2}$, and we obtain $\|u-\Tilde{u}\|_{X^{s,b}_T}\leq C\|\vartheta-\Tilde{\vartheta}\|_{H^{s}}$. Similarly, for $v-\Tilde{v}$, we have
\begin{align*}
\|v-\Tilde{v}\|_{X^{s,b}_T}&\leq C\||u|^{2}u-|\Tilde{u}|^2\Tilde{u}\|_{X^{s,-b'}_T}\\
&\leq  C\eta^2\|u-\Tilde{u}\|_{X^{s,b'}_T}\\
&\leq  C\eta^2\|\vartheta-\Tilde{\vartheta}\|_{H^{s}}.
\end{align*}
We deduce that $\|v|_{t=0}-\Tilde{v}|_{t=0}\|_{H^s}\leq \|v-\Tilde{v}\|_{X^{s,b}_T}< C\eta^2\|\vartheta-\Tilde{\vartheta}\|_{H^{s}}$.
Therefore, we obtain
\begin{align*}
\|\mathcal{J}\vartheta-\mathcal{J}\Tilde{\vartheta}\|_{H^{s}}&=\|(u_0-S_N^v\vartheta)-(u_0-S_N^v\Tilde{\vartheta}\|_{H^{s}}\\
&=\|v|_{t=0}-\Tilde{v}|_{t=0}\|_{H^s}\\
&\leq C\eta^2\|\vartheta-\Tilde{\vartheta}\|_{H^{s}}
\end{align*}
This yields that $\mathcal{J}$ is a contraction on a small ball $B_{H^s}(0,\eta)$ in $H^s$, which completes the proof.
\end{proof}
\section{Exact controllability: Proof of Theorem \ref{thm1}}\label{sec: 5}
In general, in order to Theorem \ref{thm1}, we only need to reverse time and glue the forward and backward solutions together. For the completeness of our paper, we choose to present the proof as follows.
\begin{proof}[Proof of Theorem \ref{thm1}]
We first consider the time-reversed equation (that is, the equation obtained by the change of variable $t\mapsto T-t$) of \eqref{eq: u-general-nls}:
\begin{equation*}
-\ii\p_t w+\Delta_{x,y} w+\ep|w|^{2}w=\varphi_T\chi_{\Omega}(1-\Delta_{x,y})^{-s}\varphi_T\chi_{\Omega}h,
\end{equation*}
where $w(t,x)=u(T-t,x)$. For this equation, we repeated the procedure as the proof for Theorem \ref{thm4}. There exists a control function $h$ such that for $\|w_0\|_{H^s}$ small enough, $h$ achieves null controllability for $w$. Now we choose $u_{0},u_{f}\in H^s)$, satisfying that
\begin{equation*}
 \|u_{0}\|_{H^{s}}+\|u_{f}\|_{H^{s}}<\delta,   
\end{equation*}
with $\delta$ sufficiently small. There exists $g$ such that 
\begin{equation*}
\left\{
\begin{array}{l}
\ii\p_t u+\Delta_{x,y} u+\ep|u|^{2}u=\varphi_T\chi_{\Omega}(1-\Delta_{x,y})^{-s}\varphi_T\chi_{\Omega}g\\
 u|_{t=0}=u_0   \quad u|_{t=\frac{T}{2}}=0.
\end{array}
\right.
\end{equation*}
And there exists $h$ such that 
\begin{equation*}
\left\{
\begin{array}{l}
-\ii\p_t w+\Delta_{x,y} w+\ep|w|^{2}w=\varphi_T\chi_{\Omega}(1-\Delta_{x,y})^{-s}\varphi_T\chi_{\Omega}h\\
 w|_{t=0}=u_f   \quad w|_{t=\frac{T}{2}}=0.
\end{array}
\right.
\end{equation*}
Moreover, we know that the solutions $u,w\in C([0,\frac{T}{2}],H^s)$. Now we define $U\in C([0,T],H^s)$ and $f\in C([0,T],H^s)$ by
\begin{equation*}
    U(t)=\left\{
    \begin{array}{ll}
         u(t)&t\in[0,\frac{T}{2}],  \\
         w(T-t)&t\in(\frac{T}{2},T].
    \end{array}
    \right.\quad \text{ and }\\    
\end{equation*}
\begin{equation*}
    f(t)=\left\{
    \begin{array}{ll}
         \varphi_T\chi_{\Omega}(1-\Delta_{x,y})^{-s}\varphi_T\chi_{\Omega}g,&t\in[0,\frac{T}{2}],  \\
         \varphi_T(T-t)\chi_{\Omega}(1-\Delta_{x,y})^{-s}\varphi_T(T-t)\chi_{\Omega}h(T-t),&t\in(\frac{T}{2},T].
    \end{array}
    \right.
\end{equation*}

Indeed, $f$ continues in time, since the cut-off function vanishes near $t=\frac{T}{2}$. $U$ solves the equation:
\begin{equation}
 \left\{
\begin{array}{l}
     \ii\p_t U+\Delta_{x,y} U+\ep|U|^{2}U=f, \\
     U|_{t=0}=u|_{t=0}=u_{0},\\
     U|_{t=T}=w(T-t)|_{t=T}=w|_{t=0}=u_{f}.
\end{array}
\right.   
\end{equation}
This completes the proof of Theorem \ref{thm1}.
\end{proof}

\appendix
\section{Some properties of $X^{s,b}$ spaces}\label{app}
%-------------------------------------------------------------------------
For the sake of completeness, in Appendix, we include the definitions for $X^{s,b}$ spaces in the waveguide setting, together with some useful properties. We refer to Tao \cite{Taobook} for more details. See also \cite{laurent2010global}. Since the results are well-known, we only present them and omit the proofs.

Following \cite{bourgain1993fourier,bourgain1998refinements}, one may define $X^{s,b}$ norm as
\begin{equation}
	\|u\|_{X^{s,b}}:=\| \langle \tau-|\xi|^{2}\rangle^{b}\langle \xi\rangle^{s}\tilde{u}\|_{L_{\xi,\tau}^{2}},
\end{equation}
where $u=u(z,t)$ is a function defined on $\mathbb{R}^{m}\times \mathbb{T}^{n}\times \mathbb{R}$, and $z\in \mathbb{R}^{m}\times \mathbb{T}^{n}, t\in \mathbb{R}$. And $\tilde{u}(\xi,\tau)$ is the space-time Fourier transform of $u$, where $\xi\in \mathbb{R}^{m}\times \mathbb{Z}^{n}$, $\tau\in \mathbb{R}$.

And $X^{s,b}$ spaces are just all those functions with finite $X^{s,b}$ norm.

In practice, one mainly works on $s\geq 0,b>\frac{1}{2}$.

One key property for $X^{s,b}$ space is that inherits the estimates of linear solutions, which is known to be  ``transference principle''.
 
%Since $M$ is compact, $\Delta$ has a compact resolvent and thus, the spectrum of $\Delta$ is discrete. We choose $e_k\in L^2(M)$, $k\in M$ an orthonormal basis of eigenfunctions of $-\Delta$, associated to eigenvalues $\lambda_k$. Denote $P_k$ the orthogonal projector on $e_k$. We equip the Sobolev space $H^s(M)$ with the norm (with $\left\langle x\right\rangle=\sqrt{1+|x|^2}$),
%$$\left\|u\right\|^2_{H^s(M)}=\sum_{k}\left\langle \lambda_k\right\rangle^s \left\|P_k u\right\|^2_{L^2(M)}.$$
%The Bourgain space $X^{s,b}$ is equipped with the norm
%$$\left\|u\right\|^2_{X^{s,b}}= \sum_{k}\left\langle \lambda_k\right\rangle^s \left\|\left\langle \tau+ \lambda_k \right\rangle^b \widehat{P_k}(\tau) u\right\|^2_{L^2(\R_{\tau}\times M)}=\left\|u^{\#}\right\|^2_{H^{b}(\R,H^s(M))}$$
 %where $u=u(t,x)$, $t\in\R$, $x\in M$, $u^{\#}(t)=e^{-it\Delta}u(t)$ and $ \widehat{P_k u}(\tau)$ denotes the Fourier transform of $P_k u$ with respect to the time variable.
$X^{s,b}_T$ is the associated restriction space with the norm
\bna 
\left\|u\right\|_{X^{s,b}_T}=\inf \left\{ \left\| \tilde{u}\right\|_{X^{s,b}} \left| \tilde{u}=u \textnormal{   on   } (0,T)\times M  \right. \right\}.
\ena
We also write $\left\|u\right\|_{X^{s,b}_I}$ if the   infimum is taken on functions $\tilde{u}$ equalling $u$ on an interval $I$. The following properties of $X^{s,b}_T$ spaces are easily verified.
\begin{enumerate}
 \item $X^{s,b}$ and $X^{s,b}_T$ are Hilbert spaces.
	\item If $s_1\leq s_2$, $b_1\leq b_2$ we have $X^{s_2,b_2} \subset  X^{s_1,b_1}$ with continuous embedding.
	\item For every $s_1<s_2$, $b_1<b_2$ and $T>0$, we have $X^{s_2,b_2}_T \subset  X^{s_1,b_1}_T$
with compact embedding.
\item For $0<\theta <1$, the complex interpolation space $\left(X^{s_1,b_1},X^{s_2,b_2}\right)_{[\theta]}$ is $X^{(1-\theta)s_1+\theta s_2,(1-\theta)b_1+\theta b_2}$. \label{enuminterp}
\end{enumerate}

We note that \eqref{enuminterp} can be proved with the interpolation theorem of Stein-Weiss for weighted $L^p$ spaces.

Then, we list some additional trilinear estimates that will be used all along the paper.
\begin{lemma}
\label{lmtrilinXsb}
For every $r\geq s>s_0$, there exist $0<b'<1/2$ and $C>0$ such that for any $u$ and $\tilde{u} \in X^{r,b'}$
\bnan
\label{multilinH1Hs}
 \left\| |u|^2 u \right\|_{X^{r,-b'}} &\leq& C \left\|u\right\|^2_{X^{s,b'}} \left\|u\right\|_{X^{r,b'}}\\
 \label{multilinH12Hs}
 \left\| |u|^2 \widetilde{u} \right\|_{X^{r,-b'}} &\leq& C \left\|u\right\|_{X^{s,b'}}\left\|u\right\|_{X^{r,b'}} \left\|\widetilde{u}\right\|_{X^{r,b'}}\\
\label{multilinH1}
 \left\||u|^2u-|\widetilde{u}|^2\widetilde{u}\right\|_{X^{s,-b'}} &\leq &C \left(\left\|u\right\|^2_{X^{s,b'}}+ \left\|\tilde{u}\right\|^2_{X^{s,b'}} \right) \left\|u-\tilde{u}\right\|_{X^{s,b'}}.
\enan
Moreover, the same estimates hold with $z_1\overline{z_2}z_3$ replaced by any $\R$-trilinear form on $\mathbb{C}$.
\end{lemma}
%The proof of the previous lemma can be found in \cite{InventionesBGT}, \cite{gerardcourspise} or \cite{PGPierfelice}. Yet, in the Appendix, we prove some slightly different estimates, but the proof gives an idea of how Lemma \ref{lmtrilinXsb} is established. We also give some variants that will be used in the linearized version of our equations.
\begin{lemma}
\label{lmtrilinXsbbis}
For every $-1\leq s\leq 1$ and any $s_0<r\leq 1$, there exist $0<b'<1/2$ and $C>0$ such that for any $u\in X^{s,b'}$ and $a_1,a_2 \in X^{1,b'}$
\bnan
\label{inegtrilinmult}
\left\|a_1\overline{a_2}u\right\|_{X^{s,-b'}}\leq C\left\|a_1\right\|_{X^{1,b'}} \left\|a_2\right\|_{X^{1,b'}}\left\|u\right\|_{X^{s,b'}}\\
\label{inegtrilinmultcompact}
\left\||a_1|^2 u\right\|_{X^{s,-b'}}\leq C\left\|a_1\right\|_{X^{1,b'}} \left\|a_1\right\|_{X^{r,b'}}\left\|u\right\|_{X^{s,b'}}.
\enan
Moreover, the same estimates hold with $z_1\overline{z_2}z_3$ replaced by any $\R$-trilinear form on $\C$.
\end{lemma}
%\begin{proof}
%We first prove (\ref{inegtrilinmultcompact}). Estimate (\ref{multilinH12Hs}) of Lemma \ref{lmtrilinXsb} implies that the operator of multiplication by $|a_1|^2$ maps $X^{1,b'}$ into $X^{1,-b'}$ with norm $\left\|a_1\right\|_{X^{1,b'}} \left\|a_1\right\|_{X^{r,b'}}$. IBy duality, it maps $X^{-1,b'}$ into $X^{-1,-b'}$ with the same norm. We get the same result for $-1\leq s\leq 1$ by interpolation, which yields (\ref{inegtrilinmultcompact}). For (\ref{inegtrilinmult}), we observe that estimate 
%$$\left\|a_1\overline{a_2}u\right\|_{X^{1,-b'}}\leq C\left\|a_1\right\|_{X^{1,b'}} \left\|a_2\right\|_{X^{1,b'}}\left\|u\right\|_{X^{1,b'}}$$
%holds whatever the position of the conjugate operator and we conclude similarly.
%\end{proof}
Let us study the stability of the $X^{s,b}$ spaces with respect to some particular operations.
\begin{lemma}
\label{lemmatps}
Let $\varphi \in C^{\infty}_0(\R)$ and $u\in X^{s,b}$ then $\varphi(t) u \in X^{s,b}$.\\
If $u\in X^{s,b}_T$ then we have $\varphi(t) u \in X^{s,b}_T$.
\end{lemma}
%\begin{proof}
%We write
%$$\left\|\varphi u\right\|_{X^{s,b}}= \left\|e^{-it\Delta}\varphi(t)u(t)\right\|_{H^{b}(\R,H^s)}= \left\|\varphi u^{\#}\right\|_{H^{b}(\R,H^s)}\leq C  \left\|u^{\#}\right\|_{H^{b}(\R,H^s)}\leq C\left\|u\right\|_{X^{s,b}}.$$
%We get the second result by applying the first one on any extension of $u$ and taking the infimum.
%\end{proof}
In the case of pseudo-differential operators in the space variable, we have to deal with a loss in $X^{s,b}$ regularity compared to what we could expect. Some regularity in the index $b$ is lost, due to the fact that a pseudo-differential operator does not keep the structure in time of the harmonics.

This loss is unavoidable as we can see, for simplicity on the torus $\mathbb{T}^1$ : we take  $u_n=\psi(t)e^{inx}e^{i|n^2|t}$ (where $\psi \in C^{\infty}_0$ equal to $1$ on $[-1,1]$) which is uniformly bounded in $X^{0,b}$ for every $b\geq 0$. Yet, if we consider the operator $B$ of order $0$ of multiplication by $e^{ix}$, we get $\left\|e^{ix}u_n\right\|_{X^{0,b}} \approx n^b$. Yet, we do not have such loss for operator of the form $(-\Delta)^r$ which acts from any $X^{s,b}$ to $X^{s-2r,b}$. But if we do not make any further assumption on the pseudo-differential operator, we can show that our example is the worst one : 
\begin{lemma}
\label{lemmapseudoxsb}
Let $ -1 \leq b \leq 1$ and $B$ be a pseudo-differential operator in the space variable of order $\rho$. For any $u\in X^{s,b}$ we have $B u \in X^{s-\rho-|b|,b}$.\\
Similarly, $B$ maps $X^{s,b}t$ into $X^{s-\rho-|b|,b}_T$.
\end{lemma}
We will also use the following elementary estimate.
\begin{lemma}
\label{gainint}
Let $(b,b')$ satisfying
\begin{eqnarray}
0<b'<\frac{1}{2}<b,~~~~b+b'\leq 1. 
\end{eqnarray}
If we note $F(t)=\Psi\left(\frac{t}{T}\right)\int_0^t f(t')dt'$, we have for $T\leq 1$
\bna
\left\|F\right\|_{H^b} \leq CT^{1-b-b'}\left\|f\right\|_{H^{-b'}}.
\ena
\end{lemma}
In the future aim of using a boot-strap argument, we will need some continuity in $T$ of the $X^{s,b}_T$ norm of a fixed function : 
\begin{lemma}
\label{continuiteXsbt}
Let $0<b<1$ and $u$ in $X^{s,b}$ then the function
\begin{eqnarray*}
\left\lbrace
\begin{array}{rcrcl}%argument r=alignement à droite puis center et left
f&:&(0,T]& \longrightarrow &  \R \\
 & &t    & \longmapsto     & \left\|u\right\|_{X^{s,b}_{t}}
\end{array}
\right.
\end{eqnarray*} 
 is continuous.
Moreover, if $b>1/2$, there exists $C_b$ such that 
$$\lim_{t\rightarrow 0}f(t) \leq C_b \|u(0)\|_{H^s}.$$
\end{lemma}
%\begin{proof}
%By reasoning on each component on the basis, we are led to prove the result in $H^b(\R)$. The most difficult case is the limit near $0$. It suffices to prove that if $u\in H^b(\R)$, with $b>1/2$, satisfies $u(0)=0$, and $\Psi \in C^{\infty}_0(\R)$ with $\Psi(0)=1$, then
%$$\Psi\left(\frac{t}{T}\right)u \xrightarrow{T\rightarrow0}0 \quad \textnormal{in} \quad H^b.$$
%Such a function $u$ can be written $\int_0^t f$ with $f\in H^{b-1}$. Then, Lemma \ref{gainint} gives the result we want if $u\in H^{b+\varepsilon}$. Nevertheless, if we only have $u\in H^b$, $\Psi(\frac{t}{T})u$ is uniformly bounded. We conclude by a density argument.
%\end{proof}
The following lemma will be useful to control solutions on large intervals that will be obtained by piecing together solutions on smaller ones. %We state it without proof.
\begin{lemma}
\label{lemmarecouvrement}
Let $0<b<1$. If $\bigcup (a_k,b_{k})$ is a finite covering of $[0,1]$, then there exists a constant $C$ depending only of the covering such that for every $u\in X^{s,b}$,
\begin{eqnarray*}
\left\|u\right\|_{X^{s,b}_{[0,1]}}\leq C\sum_k \left\|u\right\|_{X^{s,b}_{[a_k,b_{k}]}}.
\end{eqnarray*}
\end{lemma}

\bibliographystyle{plain}
\bibliography{ref}
\end{document}